\documentclass[11pt]{article}
\usepackage{latexsym}
\usepackage{epsfig,enumerate}
\usepackage{amsmath,amsthm,amssymb}
\usepackage{algorithm2e}
\usepackage{tikz,tikz-network}

\allowdisplaybreaks

\usepackage[margin=1in]{geometry}

\usetikzlibrary{arrows}

 \usepackage{graphicx}
 \def\E{\mathbb{E}}
\newcommand{\tbf}[1]{\textbf{#1}}

\newtheorem{theorem}{Theorem}[section]
\newtheorem{lemma}[theorem]{Lemma} 
\newtheorem{proposition}[theorem]{Proposition} 
\newtheorem{corollary}[theorem]{Corollary}

{\theoremstyle{definition}

\newtheorem{definition}{Definition}

 }

\newcommand{\wo}{\setminus} 
\newcommand{\up}[1]{\left\lceil{#1}\right\rceil} 
 
\newcommand{\abs}[1]{\left|{#1}\right|}

\newcommand{\set}[1]{\left\{{#1}\right\}} 
\newcommand{\setof}[2]{\left\{{#1}\,:\,{#2}\right\}} 
\newcommand{\of}{\subseteq}
\newcommand{\op}[1]{\left({#1}\right)}

\newcommand{\symd}{\bigtriangleup}
\newcommand{\adj}{\sim}
\newcommand{\nadj}{\not\adj}

\newcommand{\Gbar}{\overline{G}}
\newcommand{\Gxy}{G_{x\to y}}
\newcommand{\I}{\mathcal{I}}
\newcommand{\Tbar}{\overline{T}}

\AtBeginDocument{%
   \def\MR#1{}
}

\title{Nordhaus-Gaddum inequalities for the number of cliques in a graph}

\author{Deepak Bal\thanks{Department of Mathematics, Montclair State University, Montclair, NJ, 07043 U.S.A. \texttt{deepak.bal@montclair.edu}}\and  Jonathan Cutler\thanks{Department of Mathematics, Montclair State University, Montclair, NJ, 07043 U.S.A. \texttt{jonathan.cutler@montclair.edu}}\and Luke Pebody\thanks{\texttt{luke@pebody.org}}}
\date{}
\begin{document}
\maketitle

\begin{abstract}
	Nordhaus and Gaddum proved sharp upper and lower bounds on the sum and product of the chromatic number of a graph and its complement.  Over the years, similar inequalities have been shown for a plenitude of different graph invariants.  In this paper, we consider such inequalities for the number of  cliques (complete subgraphs) in a graph $G$, denoted $k(G)$.  We note that some such inequalities have been well-studied, e.g., lower bounds on $k(G)+k(\overline{G})=k(G)+i(G)$, where $i(G)$ is the number of independent subsets of $G$, has been come to be known as the study of Ramsey multiplicity.  We give a history of such problems.  One could consider fixed sized versions of these problems as well.  We also investigate multicolor versions of these problems, meaning we $r$-color the edges of $K_n$ yielding graphs $G_1,G_2,\ldots,G_r$ and give bounds on $\sum k(G_i)$ and $\prod k(G_i)$.  
\end{abstract}

\section{Introduction}

Extremal results in graph theory often involve maximizing or minimizing some graph invariant over graphs on a fixed number of vertices with some sort of restriction on the number of edges.  The restriction on the number of edges can be dispensed with if one considers both the graph and its complement simultaneously.  Nordhaus and Gaddum \cite{NG} gave upper and lower bounds on the sum and product of the chromatic numbers of a graph on $n$ vertices and its complement.

\begin{theorem}[Nordhaus, Gaddum 1956]
	If $G$ is a graph on $n$ vertices, then
	\[
		2\sqrt{n}\leq \chi(G)+\chi(\Gbar)\leq n+1,
	\]
	and
	\[
		n\leq \chi(G)\chi(\Gbar)\leq \left(\frac{n+1}2\right)^2.
	\]
\end{theorem}

In fact, as Nordhaus and Gaddum noted, the lower bound on $\chi(G)\chi(\Gbar)$ was already known as it had been proved by Zykov \cite{Z}.  Prompted by this result, there has been an abundance of papers that have studied Nordhaus-Gaddum inequalities for a variety of graph invariants.  (See \cite{AH} for a survey.)

More recently, there has been interest in studying Nordhaus-Gaddum inequalities for the \emph{number} of certain sets in a graph.  For example, Wagner \cite{W} gave the following lower bound on $\partial(G)+\partial(\Gbar)$, where $\partial(G)$ is the number of dominating sets in a graph $G$.

\begin{theorem}[Wagner 2013]
	If $G$ is a graph on $n$ vertices, then
	\[
		\partial(G)+\partial(\Gbar)\geq 2^n.
	\]
\end{theorem}

Keough and Shane \cite{KS}, prompted by a question of Wagner, proved an upper bound on $\partial(G)+\partial(\Gbar)$ that is sharp in the lead term.  

We will be interested in Nordhaus-Gaddum-type inequalities for $i(G)$, the number of independent sets in $G$, and $k(G)$, the number of cliques (complete subgraphs) in $G$.  We also let $i_t(G)$ be the number of independent sets of size $t$ in $G$ and, likewise, $k_t(G)$ be the number of cliques of size $t$ in $G$.  It should be noted that $i_0(G)=k_0(G)=1$ for any graph $G$ and $i_1(G)=k_1(G)=n$ for any graph $G$ on $n$ vertices.

To introduce some (more) notation, we are interested in studying the behavior of $\sigma(G)=i(G)+i(\Gbar)=i(G)+k(G)$ and $\pi(G)=i(G)i(\Gbar)=i(G)k(G)$ as $G$ ranges over all graphs on $n$ vertices.  We are also interested in fixed size versions of these quantities and so let $\sigma_{t}(G)=i_{t}(G)+k_{t}(G)$ and $\pi_{t}(G)=i_{t}(G)k_{t}(G)$.  Bounds on these have been studied in various other contexts and we will present those via Nordhaus-Gaddum inequalities.  Hu and Wei \cite{HW} gave the following upper bound on $\sigma(G)$.

\begin{theorem}[Hu, Wei 2018]\label{thm:HW}
	If $G$ is a graph on $n$ vertices, then
	\[
		\sigma(G)\leq 2^n+n+1,
	\]
	with equality if and only if $G$ is $K_n$ or $E_n$.
\end{theorem}

In fact, Hu and Wei were more interested in studying Nordhaus-Gaddum inequalities for connected graphs, and Theorem~\ref{thm:HW} was a corollary of their main result.  We note, however, that this result is, in fact, rather easy to prove as every subset of a graph can be either empty or complete, but not both unless it is the empty set or a single vertex.  The only graphs in which every subset of the vertex set is either complete or empty (or both) are $K_n$ and $E_n$.

As for a lower bound on $\sigma(G)$, it should be noted this problem is equivalent to finding the minimum number of complete monochromatic subgraphs in a $2$-edge-coloring of $K_n$.  This is a well-studied problem that is known as \emph{Ramsey multiplicity}. We use $\log$ to represent the natural logarithm throughout this paper.
Feige, Kenyon and Kogan proved the following in \cite{FKK}.
\begin{theorem}[Feige et al. 2020]
	If $G$ is a graph on $n$ vertices, then
	\[
		\sigma(G)\geq n^{(\frac14 +o(1)) \log_2 n}\ge n^{0.36\log n}.
	\]
\end{theorem}
This was an improvement over a result of Sz\'ekely \cite{Sz} from 1984 which showed that $\sigma(G)\geq n^{0.2275\log n}$ On the other hand, the random graph shows that this bound cannot be improved beyond $n^{(\frac12 +o(1))\log_2 n}\approx n^{0.72\log n}$. See the proof of Proposition \ref{prop:rgtight} for details.



The first of the main result of this paper is the following, which bounds $\pi(G)$ both above and below.

\begin{theorem}\label{thm:pi}
For any graph $G$ on $n$ vertices \[ n^{\op{\frac{1}{2}+o(1)}\log_2 n} \le \pi(G) \le (n+1)2^n.\] Furthermore, the upper bound is attained by $K_n$ and $E_n$. 
\end{theorem}

We prove the upper bound in Section~\ref{sec:upperpi} using compression.  The proof of the lower bound uses main idea of the proof of Feige et al.~\cite{FKK} which appears in Section~\ref{sec:lowerpi}.  The random graph $G(n,1/2)$ shows that the lower bound cannot be improved beyond $n^{(1+o(1))\log_2 n}$.

\begin{proposition}\label{prop:rgtight}Let $G\sim G(n,1/2)$. Then
$\pi(G) \le n^{(1+o(1))\log_2 n}$ with high probability.
\end{proposition}
\begin{proof}


Note that $G$ and $\Gbar$ are both distributed as $G(n,1/2)$. The expected number of cliques in $G$ is bounded by 
\[\sum_{i\ge 0} \binom{n}{i}2^{ - \binom{i}{2}} \le \sum_{i\ge 0} \op{ \frac{ne\sqrt{2}}{i\cdot2^{i/2}} }^i  \]
where we used the estimate $\binom{n}{i} \le \op{\frac{ne}{i}}^i$. Let $i_*$ be the (real) value of $i$ which maximizes the summand on the right hand ride. Using calculus, we find that $i_*$ satisfies $n\sqrt{2} = i_*r^{i_*}$ (i.e.,  $i_*\sim \log_2\op{ \frac{n\sqrt 2}{\log_2 n}} \sim \log_2 n$). Thus we have 
\[\E [k(G)] \le n\cdot (e2^{i_*/2})^{i_*} = n^{\op{\frac12 + o(1)} \log_2 n } \]
So using Markov's inequality, we have that w.h.p.\ $k(G) \le n^{\op{\frac12 + o(1)} \log_2 n}$ and so  by the union bound (over $G$ and $\Gbar$), we have that w.h.p.
\[
	\pi(G) \le n^{\op{1 + o(1)}\log_2 n}.
\]
\end{proof}

One can also consider fixed size versions of these problems.  So, using the notation introduced above, we would like to find upper and lower bounds for $\sigma_{t}(G)=i_{t}(G)+k_{t}(G)$ and $\pi_{t}(G)=i_{t}(G)k_{t}(G)$.  Two of these bounds are, in fact, trivial.  Note that if $G$ is a graph on $n$ vertices and $t\geq 2$, then
\[
	\sigma_{t}(G)\leq \binom{n}{t}\qquad\text{and}\qquad 0\leq \pi_{t}(G).
\]
Also, both inequalities are sharp for $G=E_n$ or $K_n$.  (Note that $\sigma_0(G)=2$, $\pi_0(G)=1$, $\sigma_1(G)=2n$, and $\pi_1(G)=n^2$ for all $G$ on $n$ vertices.)

The other two bounds are much more interesting.  The lower bound on $\sigma_{t}(G)$ is a well-studied and difficult problem under the umbrella of Ramsey multiplicity.  The history begins with Goodman~\cite{G59} who gave a lower bound for $t=3$.  Improving a result of Erd\H{o}s~\cite{E62}, Conlon~\cite{C} proved the following.

\begin{theorem}[Conlon 2012]
	If $n$ is large enough and $G$ is a graph on $n$ vertices, then 
	\[
		\sigma_{t}(G)\geq \frac{n^{t}}{C^{(1+o(1))t^2}},
	\]
	where $C\approx 2.18$ is an explicit constant. 
\end{theorem}

For an upper bound on $\pi_{t}(G)$, we note that a related problem was studied (independently) by Huang et al.~\cite{HLNPS} and Frankl et al.~\cite{FKSSW}.  Both of these groups were interested in finding 
\[
	\max\setof{\min\set{k_{t}(G),i_{t}(G)}}{\text{$G$ is a graph on $n$ vertices}}.
\]
In fact, Huang et al.~were able to do a bit more as the size of the complete graphs and independent sets could be different.  In Section~\ref{sec:pifix}, we find a lower bound on $\pi_{t}(G)$.  While we are able to use some elements of the techniques of Frankl et al., the proof requires quite a few new ideas.  Our result is as follows.

\begin{theorem}\label{thm:pil}
Let $t\ge 3$ be fixed and define $\mu_t = \frac14\frac{t-2 + \sqrt{t^2 + 4t - 4}}{t - 1}$
and 
\[
	f_t(x) = x^t(1-x)^{t-1}(1 + (t-1)x).
\]
For any graph $G$ on $n$ vertices, 
\[\
	\pi_t(G) = k_t(G)i_t(G) \le \op{\frac{n^t}{t!}}^2f_t(\mu_t) + O(n^{2t-1}).
\]
Moreover, this bound is tight as seen by considering the threshold graph with code $(+)^{(1-\mu_t) n} (-)^{\mu_t n}$ or with code $(-)^{(1-\mu_t)n}(+)^{\mu_t n}$. 
\end{theorem}

Prompted by a question of Pfender \cite{Pf}, we also consider multicolored generalizations of these problems.  To be precise, we color the edges of $K_n$ with, say, $r$ colors, yielding graphs $G_1, G_2, \ldots, G_r$.  We can then attempt to bound
\[
	\sum_{i=1}^r k(G_i)\qquad\text{and}\qquad \prod_{i=1}^r k(G_i).
\]
When $r=2$, these of course reduce to the problems outlined above.  The upper bound on the sum is once again straightforward, generalizing Theorem~\ref{thm:HW}.

\begin{proposition}
Let $r\ge 2$. Suppose $K_n$ is $r$-colored and that $G_i$ is the graph of color $i$ for each $i\in [r]$. Then $\sum_{i=1}^r k(G_i) \le (r-1)(n+1) + 2^n$. Moreover, this bound is achieved when $G_i = K_n$ for some $i\in [r]$. {}
\end{proposition}

\begin{proof}
Consider any $A\subseteq [n]$ with $|A|\ge 2.$ Then $A$ can be a clique in at most one of the colors. Thus the total number of monochromatic cliques of size at least 2 is at most $2^n - n - 1$. On the other hand, when $|A|=0$ or when $|A|=1$, then $A$ is considered a clique in every color. Thus we have $\sum_{i=1}^r k(G_i) \le r(n+1) + 2^n - n - 1 = (r-1)(n+1) + 2^n.$  The only way to achieve equality is for every subset of the vertices to yield a monochromatic clique, which can only happen when the entire graph is monochromatic.
\end{proof}

A lower bound on the product can be obtained by modifying the technique of Feige et al.~\cite{FKK}, which is done in Section~\ref{sec:lowerpi}. A lower bound on the sum follows from this as a simple corollary.  Perhaps most interestingly, the upper bound on the product does not follow from the same argument as the two-color case as compression no longer works. One may still suspect that, as in the sum case, the construction where one color is complete and the remaining colors are empty gives the optimal bound, but we show that this is not the case.
We are able to give an upper bound that is sharp up to a constant multiple depending on the number of colors.  This is done in Section~\ref{sec:train}.
Our results for $r$ colors are summarized in the following theorem.

\begin{theorem}\label{thm:multi-main}
Let $r\ge 3$. Suppose the edges of $K_n$ have been $r$-colored and for all $i\in [r]$,  $G_i$ represents the graph of $i$-colored edges. Then there is a constant $C_r$ such that

\[n^{(\frac12 +o(1))\log_r n}\le\prod_{i=1}^r k(G_i) \le C_r n^{\stackrel{\binom r2}{}}2^n.\]
 Furthermore, the upper bound is tight up to the constant $C_r$ and the lower bound is tight up to a multiple of $3\log_3 r$ in the exponent.
\end{theorem}

\section{Upper bound on $\pi(G)$}\label{sec:upperpi}

One tool that helps us in some cases is that of compression.  Compression has been used many times to prove various results in extremal graph theory.  Given two vertices $x$ and $y$ in a graph $G$, they partition the rest of the vertices into four sets:
\begin{align*}
	N_x(G)&=\setof{v\in V(G)\wo \set{x,y}}{v\adj x, v\nadj y},\\
	N_y(G)&=\setof{v\in V(G)\wo \set{x,y}}{v\nadj x, v\adj y},\\
	N_{xy}(G)&=\setof{v\in V(G)\wo \set{x,y}}{v\adj x, v\adj y},\text{ and}\\
	O_{xy}(G)&=\setof{v\in V(G)\wo \set{x,y}}{v\nadj x, v\nadj y}.
\end{align*}
We then define the \emph{compression of $G$ from $x$ to $y$}, denoted $\Gxy$, to be the graph with vertex set $V(G)$ and edge set consisting of all edges of $G$ except those in $N_x$ with the addition of edges between $y$ and $N_x$.  In other words, $N_x(\Gxy)=\emptyset$, $N_y(\Gxy)=N_x(G)\cup N_y(G)$, $N_{xy}(\Gxy)=N_{xy}(G)$, and $O_{xy}(\Gxy)=O_{xy}(G)$.  The following lemma is well-known (see, e.g., \cite{CR}), but we include a sketch of a proof for completeness.

\begin{lemma}\label{lem:comp}
	If $G$ is a graph with vertices $x$ and $y$, then $i(G)\leq i(\Gxy)$ and $i(\Gbar)\leq i(\Gbar_{x\to y})$.
\end{lemma}

\begin{proof}
	We show that there is an injection from $\I(G)\wo \I(\Gxy)$ to $\I(\Gxy)\wo \I(G)$.  Note that if $I\in \I(G)\wo \I(\Gxy)$, then it must contain an edge of $\Gxy$ that does not exist in $G$.  Thus, $y\in I$ and there is some $z\in N_x(G)$ such that $z\in I$.  But then $x\not\in I$ since $xz\in E(G)$ and so $I\symd \set{x,y}\in \I(\Gxy)\wo \I(G)$.  It is straightforward to check that the map $I\mapsto I\symd \set{x,y}$ is an injection.  Thus, we have that $i(G)\leq i(\Gxy)$.
	
	For the other inequality, note that compression in $G$ yields compression in $\Gbar$ as well.  In particular, $\overline{\Gxy}=\Gbar_{y\to x}$.  Thus, we also have that $i(\Gbar)\leq i(\Gbar_{x\to y})$.
\end{proof}

In fact, the same proof yields that the number of independent sets of fixed size does not decrease under compression.

\begin{corollary}
	If $G$ is a graph with vertices $x$ and $y$, then $i_{t}(G)\leq i_{t}(\Gxy)$ and $i_{t}(\Gbar)\leq i_{t}(\Gbar_{x\to y})$.
\end{corollary}

\begin{proof}
	Note that in the proof of Lemma~\ref{lem:comp}, the map $I\mapsto I\symd\set{x,y}$ is in fact a map between sets of the same size.  The result follows.
\end{proof}

Note that compressions force (closed) neighborhoods to become nested, and so repeated compressions yield a graph in which the closed neighborhoods are ordered by containment.  Such graphs are known as \emph{threshold graphs}.

\begin{definition}
	A graph is a \emph{threshold graph} if it can be formed iteratively from a single vertex by, at each step, adding either an isolated vertex or a dominating vertex.
\end{definition}

A threshold graph on $n$ vertices, by definition, can be represented by a binary code of length $n-1$ (the initial vertex does not need to be encoded) where $+$ corresponds to a dominating vertex and $-$ corresponds to an isolate.  We call this the \emph{code} of a threshold graph and write it from right to left.

\begin{lemma}\label{lem:thres}
For any graph $G$ on $n$ vertices, there is a threshold graph $T$ on $n$ vertices such that  
\[\pi(G) \le \pi(T) \quad\text{and}\quad \pi_t(G) \le \pi_t(T).\]
\end{lemma}


\begin{theorem}
	If $G$ is a graph on $n$ vertices, then
	\[
		\pi(G)\leq (n+1)2^n=\pi(K_n).
	\]
\end{theorem}

\begin{proof}
	We prove the statement by induction on $n$ and note that the statement is trivial for $n=1$.  Thus, assume the statement is true for $n$ and let $G$ be a graph on $n+1$ vertices.  By Lemma~\ref{lem:thres}, we know that there is a threshold graph $T$ on $n+1$ vertices with $\pi(G)\leq \pi(T)$.  Note that if $T$ is a threshold graph, then $\Tbar$ is also a threshold graph (with complementary code to that of $T$).  Further, we know that $\pi(T)=\pi(\Tbar)$ and so may assume that the final vertex, say $x$, added to $T$ was a dominating vertex.  Let $T'=T-x$.  Then, since $x$ is a dominating vertex, it can be added (or not) to any clique of $T'$ to get a clique in $T$.  Also, the vertex $x$ is only in one independent set in $T$, namely $\set{x}$.  So, we have $\pi(T)=i(T)k(T)=(i(T')+1)2k(T')$.  So,
	\begin{align*}
		\pi(T)&=(i(T')+1)2k(T')\\
		&=2\pi(T')+2k(T')\\
		&\leq 2[(n+1)2^n]+2\cdot 2^n\\
		&=(n+2)2^{n+1},
	\end{align*}
	where the inequality follows from the induction hypothesis and the fact that $k(T')\leq 2^n$.  
\end{proof}

\section{Lower bound for $\pi(G)$}\label{sec:lowerpi}

In this section, we prove the lower bounds in Theorem~\ref{thm:pi} and Theorem~\ref{thm:multi-main}.

\begin{theorem}\label{thm:r-col-lb}
Let $r\ge2$ and suppose that $r^m \le n < r^{m+1}$.  Suppose $K_n$ is $r$-colored and that $G_i$ is the graph of color $i$ for each $i\in [r]$. Then \[\prod_{i=1}^r k(G_i) \ge r^{\frac{1}{2}m^2 - O(m\log m)} = n^{(\frac12 + o(1))\log_r n}\]
\end{theorem}
\begin{proof}

As noted in the introduction, we use a modification of a technique of Feige et al.~\cite{FKK}. 
Let $G$ be an $r$-colored $K_n$. We say that a sequence of distinct vertices $(v_1, \ldots, v_q)$ is \emph{good} if for all $i=1,\ldots q-1$, there exists $C_i\in[r]$ such that $\set{v_{i+1},\ldots, v_q} \of N_{C_i}(v_i)$. Here, $N_{C}(v)$ refers to the $C$-colored neighborhood of $v$, i.e. $\set{w\in G\,:\, vw \textrm{ is color } C}$.

Let $\mathcal{X}(G,q)$ be the set of all good sequences of $G$ of length $q$.
Define the recursive sequence $a_1 = n$ and  $a_{k+1} = \up{\frac{a_k  -1}{r}}$ for all $k\ge 1$. Note that since $n\ge r^m$, we have that $a_k \ge r^{m-k+1}$ for all $1\le k\le m$. Indeed, by induction, we have that  $a_{k+1} = \up{\frac{a_k -1}{r}} \ge \up{\frac{r^{m-k+1}-1}{r}} = r^{m-k}
$.

We now claim that 
\begin{equation}\label{eq:chitwosides}
\prod_{i=1}^q a_i \le \mathcal{X}(G,q)\le q!\cdot \prod_{i=1}^r k(G_i) 
\end{equation}
To see the lower bound, we note that there are $a_1 = n$ choices for the first vertex, $v_1$, in a good sequence. Now $v_1$ has at least $a_2$ neighbors of some color $C_1\in[r]$ and so there are at least $a_2$ choices for $v_2$. We restrict $G$ to $N_{C_1}(v_1)$ and note that $v_2$ has at least $a_3$ neighbors in $N_{C_1}(v_1)$ of some color $C_2\in [r]$. Continuing in this manner (restricting to the nested monochromatic neighborhoods at each step) we have the lower bound.

To see the upper bound, notice that all the vertices $v_i$ in a good sequence with $C_i = k$ form a clique in color $k$. So each good sequence corresponds to a tuple of monochromatic cliques $(A_1,\ldots,  A_r)$. A set of vertices $\set{v_1,\ldots, v_q}$ can appear in at most $q!$ many good sequences.
Rearranging \eqref{eq:chitwosides} and setting $q=m$, we have
\[\prod_{i=1}^r k(G_i) \ge \frac{1}{m!}\prod_{i=1}^{m}a_i\ge \frac{1}{m!}\prod_{i=1}^{m}r^{m-i+1}\ge r^{\frac12 m^2 - O(m\log m)}.\]
\end{proof}

The tightness mentioned in Theorem \ref{thm:multi-main} 
follows from considering a randomly $3$-colored $K_n$. A nearly identical proof to that of 
 Proposition \ref{prop:rgtight} 
 shows that in this case, with high probability, \[\prod_{i=1}^3 k(G_i) \le n^{(\frac32 +o(1))\log_3 n} =  n^{(3\log_3r)\cdot(\frac12+o(1)) \log_r n}.\]
 It is interesting to note that in this case, taking a randomly $r$-colored $K_n$ would lead to an exponent of essentially $\frac r2 \log_r n$ which is larger than the given exponent for $r=2$ and all $r>3$.

We have the following Corollary which provides a lower bound for $\sum_{i=1}^r k(G_i)$ for $r\ge 2$. The result follows simply by applying the AM-GM inequality $\sum_{i=1}^r k(G_i) \ge r\cdot \op{\prod_{i=1}^r k(G_i)}^{1/r}$. 

\begin{corollary}
Let $r\ge2$ and suppose that $r^m \le n < r^{m+1}$.  Suppose $K_n$ is $r$-colored and that $G_i$ is the graph of color $i$ for each $i\in [r]$. Then \[\sum_{i=1}^r k(G_i) \ge r^{\frac{1}{2r}m^2 - O(m\log m)} = n^{(\frac{1}{2r} + o(1))\log_r n}.\]
\end{corollary}

\section{Upper bound on $\pi_{t}(G)$}\label{sec:pifix}

In this section we prove Theorem \ref{thm:pil}. 
As noted in the introduction, to find an upper bound on $\pi_{t}(G)=i_{t}(G)k_{t}(G)$, we begin by following the method of Frankl et al.~\cite{FKKT}. We give a moderately terse description of the setup and refer the reader to Section 2.3 of \cite{FKKT} for more details.

Firstly, we note that by Lemma~\ref{lem:thres}, we need only to maximize $\pi_{t}(T)$ among threshold graphs.  As noted above, a threshold graph $T$ on $n$ vertices is formed by successively adding dominating and isolated vertices and so can be associated with a binary code where $+$ denotes an added dominating vertex and $-$ denotes an isolate.  The set of all dominating vertices forms a clique, say $V_K$, in $T$ and the set of all isolates forms an independent set, say $V_I$.  Note that $V(T)$ is the disjoint union of $V_I$ and $V_K$ and there is, of course, a bipartite graph between $V_I$ and $V_K$.  If we want to count the number of cliques of size $t$ in $T$, then we get
\[
	S_K:=k_{t}(T)=\binom{\abs{V_K}}{t}+\sum_{w\in V_I} \binom{d_T(w)}{t-1}.
\]
Likewise, if we want to count the number of independent sets of size $t$, we get
\[
	S_I:=i_{t}(T)=\binom{\abs{V_I}}{t}+\sum_{v\in V_K} \binom{d_{\overline{T}}(v)}{t-1}.
\]

\newcommand{\vkb}{r}
\newcommand{\vib}{s}
\newcommand{\tbar}{\overline{T}}
\renewcommand{\t}{\tau}

We can order the vertices in $V_K$ and $V_I$ according to their degree in $\tbar$ and $T$, respectively (note that we are effectively ignoring the edges within $V_K$ since we only consider $d_{\tbar}(v)$ for all $v\in V_K$).   We let $V_K=\set{v_1,v_2,\ldots,v_\vkb}$ where $d_{\tbar}(v_i)=a_i$ and $a_1\geq a_2\geq \cdots\geq a_\vkb$.  Also, $V_I=\set{w_1,w_2,\ldots,w_\vib}$ where $d_{T}(w_j)=b_j$ and $b_1\leq b_2\leq \cdots\leq b_\vib$.  Thus, we have
\[
	S_K=\frac{1}{t !}\left(\vkb^{t}+t\sum_{j=1}^\vib b_j^{t-1}+O(n^{t-1})\right)\qquad\text{and}\qquad S_I=\frac{1}{t !}\left(\vib^{t}+t\sum_{i=1}^\vkb a_i^{t-1}+O(n^{t-1})\right).
\]
\newcommand{\ba}{\tbf{a}}
\newcommand{\bb}{\tbf{b}}
\newcommand{\bc}{\tbf{c}}
So, we are interested in maximizing the product $S_KS_I$.  In order to do this, we would like to shift to a continuous version of the problem.  However, we must first limit the types of sequences $\ba=(a_1,a_2,\ldots,a_\vkb)$ and $\bb=(b_1,b_2,\ldots,b_\vib)$.  We can do this using the well-known Gale-Ryser Theorem~\cite{G,R} that classifies the degree sequences of bipartite graphs.  Recall that given a sequence $\bc=(c_1,c_2,\ldots,c_s)$, we obtain its \emph{conjugate} $\bc^*=(c_1^*,c_2^*,\ldots)$ by defining $c_j^*=\abs{\set{i:c_i\geq j}}$.  Also we say that a sequence $\ba$ is \emph{majorized} by $\bc$, denoted $\ba\prec \bc$, if $\sum_{i=1}^k a_i\leq \sum_{i=1}^k c_i$ for all $k$ (where we extend one of the sequences by appending $0$s if necessary).

\begin{theorem}[Gale, Ryser]
	Let $\ba=(a_1,a_2,\ldots,a_r)$ and $\bc=(c_1,c_2,\ldots,c_s)$ be non-increasing sequences of non-negative integers with the same sum.  Then there is a bipartite graph $G$ with partition $V(G)=A\cup C$ such that $\ba$ and $\bc$ are the degree sequences of $A$ and $C$ if and only if $\ba\prec \bc^*$.
\end{theorem}

Note that we cannot apply the Gale-Ryser Theorem to our sequences $\ba$ and $\bb$ above since $\ba$ is a degree sequence in $\tbar$ and $\bb$ is a degree sequence in $T$.  But if we let $c_j=\vkb-b_j$, then $\bc=(c_1,c_2,\ldots,c_\vib)$ is the degree sequence of $V_I$ in $\tbar$ and, further, $c_1\geq c_2\geq \cdots \geq c_\vib$.  Since $\sum_{i=1}^\vkb a_i=\sum_{j=1}^\vib c_j$, the Gale-Ryser Theorem yields that $\ba\prec \bc^*$.

Since $a_i^{t-1}$ is a convex function, we note that if $\ba\prec \hat{\ba}$, then $\sum a_i^{t-1}\leq \sum \hat{a}_i^{t-1}$.  Suppose that we fix the sequence $\bb$ so that $S_K$ is fixed.  Then we'd like to maximize $S_I$ or, equivalently, maximize $\sum a_i^{t-1}$.  But since $\ba\prec \bc^*$, we need to take $\ba=\bc^*$. 
If $\ba=(a_1,a_2,\ldots,a_\vkb)$ and $\bb=(b_1,b_2,\ldots,b_\vib)$, where $\vib\geq a_1\geq a_2\geq\cdots a_\vkb\geq 0$ and $0\leq b_1\leq b_2\leq\cdots\leq b_\vib\leq \vkb$, then we say $\ba$ and $\bb$ are \emph{packed} if $\ba=\bc^*$ where $\bc=(\vkb-b_1,\vkb-b_2,\ldots,\vkb-b_\vib)$.  We are interested in 
\[
	g_{t}(\vkb,\vib)=\frac{1}{(t !)^2}\max \setof{\left(\vkb^{t}+t\sum_{j=1}^\vib b_j^{t-1}\right)\left(\vib^{t}+t\sum_{i=1}^\vkb a_i^{t-1}\right)}{\textup{$\ba$ and $\bb$ are packed}},
\]
where the $\ba$ and $\bb$ in the maximization problem are as above.  Lastly, we let
\[
	g_{t}(n)=\max\setof{g_{t}(\vkb,\vib)}{1\leq \vkb\leq \vib\leq n, \vkb+\vib=n}.
\]
We have that for any $G$ on $n$ vertices, $\pi_t(G)\le g_t(n)$.
The packing requirement on $\ba$ and $\bb$ means that if we are given $\vkb$, $\vib$, and $\bb$, then $\ba$ is determined.  Further, we can visualize this packing by noting that $\ba$ and $\bb$ can be packed into an $\vkb\times \vib$ rectangle of unit squares where the border between the $\ba$ squares and the $\bb$ squares yields an ``up-right'' path from $(0,0)$ to $(\vib,\vkb)$.  
As an example, 
let $\vib=4, \vkb=3$ and $\bb= (0,1,1,3)$. Then $\bc = (3, 2, 2, 0)$ and so $\ba = \bc^* = (3,3,1)$. This example is visualized in Figure \ref{fig:pack}.

\begin{figure}[h]
\begin{center}
\includegraphics[scale=1]{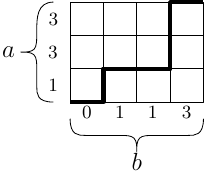}
\caption{Packing $\ba=(3,3,1)$ and $\bb=(0,1,1,3)$ into a $3\times 4$ rectangle.} \label{fig:pack}
\end{center}
\end{figure}

These paths  generalize easily to a continuous setting.
We let $p,q\in \mathbb{R}^+$ be such that $p+q=1$ (interpreting $p$ as $r/n$ and $q$ as $s/n$). Let $x:[0,1]\to [0,q]$ and $y:[0,1]\to [0,p]$ be differentiable.  The path $\ell=\setof{(x(\t),y(\t))}{0\leq \t\leq 1}$ is said to be a \emph{border} if 
\begin{enumerate}
	\item $\ell$ starts at $(0,0)$ and ends at $(q,p)$, i.e., $x(0)=y(0)=0$, $x(1)=q$, and $y(1)=p$,
	\item $\ell$ is non-decreasing, i.e., $x'(\t)\geq 0$ and $y'(\t)\geq 0$ for all $\t\in [0,1]$, and
	\item each segment of $\ell$ is either a horizontal or a vertical line, i.e., $x'(\t)y'(\t)=0$ for all $\t\in [0,1]$. 
\end{enumerate}
Let $\mathcal{L}=\mathcal{L}(q)$ be the set of all borders, and let $\mathcal{L}=\cup_{i\geq 1} \mathcal{L}_i$, where $\mathcal{L}_i$ is the set of borders with exactly $i$ turns.  Finally, define
\begin{equation}\label{eq:hqdef}
	h_{t}(q)=\max_{\ell\in \mathcal{L}} \left(p^{t}+t I_X(\ell)\right)\left(q^{t}+t I_Y(\ell)\right),
\end{equation}
where 
\[
	I_X(\ell)=I_X=\int_0^1 (y(\tau))^{t-1}x'(\tau)\ d\tau\quad\text{and}\quad I_Y(\ell)=I_Y=\int_0^1 (x(\t))^{t-1} y'(\t)\ d\t\
\]
for $\ell=\setof{(x(\t),y(\t))}{0\leq \tau\leq 1}\in \mathcal{L}$.
So, we have
\[
	I_X(\ell)\leq qp^{t-1}\quad\text{and}\quad I_Y(\ell)\leq pq^{t-1},
\]
for all $\ell\in \mathcal{L}$. We then have that for any $G$ on $n$ vertices, 
\begin{equation}\label{eq:pitubscaled}
\pi_t(G)\le \op{\frac{n^{t}}{t!}}^2\max_{q\in [0,1]}h_t(q) + O(n^{2t-1}).
\end{equation}
Our goal is to show that, in fact, that the maximum in  \eqref{eq:hqdef} occurs in $\mathcal{L}_1$ (which contains only two paths: ``up-right'' and ``right-up.''). These paths correspond to threshold graphs with only one sign change in their code. Once we know this, the maximization in \eqref{eq:pitubscaled} becomes a simple single variable calculus problem. 
Crucially, the following result of Frankl et al.~\cite{FKKT}, which is Lemma 6 in their paper, tells us that we can already restrict our attention to borders in $\mathcal{L}_1\cup \mathcal{L}_2$.

\begin{theorem}[Frankl, Kato, Katona, Tokushige \cite{FKKT}]
	Let $n\geq 3$.  For every $\ell\in \mathcal{L}_n$, there is an $\ell'\in \mathcal{L}_{n'}$ with $n'<n$ such that
	\[
		I_X(\ell')>I_X(\ell)\qquad\text{and}\qquad I_Y(\ell')>I_Y(\ell).
	\]
\end{theorem}

Further, paths in $\mathcal{L}_2$ either start going up, turn right, go all the way across, and finally turn up to end at $(q,p)$, or they start going right, turn up, go all the way up, and turn right to end at $(q,p)$.  By symmetry, we can analyze the former case which is pictured in Figure~\ref{fig:abc}.

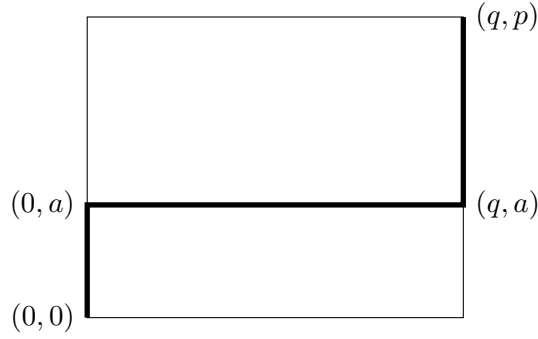
\begin{figure}[h]
\begin{center}
\begin{tikzpicture}
	\draw (0,0) rectangle (5,4);
	\draw[line width=2pt] (0,0) node[anchor=east] {$(0,0)$} -- (0,1.5) node[anchor=east] {$(0,a)$} -- (5,1.5) node[anchor=west] {$(q,a)$} -- (5,4) node[anchor=west] {$(q,p)$};
\end{tikzpicture}
\caption{An example of a path $\ell$ in $\mathcal{L}_2$ with turn points $(0,a)$ and $(q,a)$.}\label{fig:abc}
\end{center}
\end{figure}
Note that, using $(0,a)$ as the point at which the border turns right as in Figure~\ref{fig:abc}, we have
\[
	I_X(\ell)=qa^{t-1}\qquad\text{and}\qquad I_Y(\ell)=(p-a)q^{t-1}.
\]
We let $b=p-a$, so that $p=a+b$, and let $q=c$.  Then by \eqref{eq:hqdef}, we are trying to maximize
\begin{equation}
	\left((a+b)^{t}+t ca^{t-1}\right)\left(c^{t}+t bc^{t-1}\right)\label{eqn:prod}
\end{equation}
subject to $a,b,c\geq 0$ and $a+b+c=1$. Note that if $c=0$, then this product is 0, so we will assume below that $c\neq 0$.
We would like to show that the maximum is achieved on the boundary (when either $a=0$ or $b=0$).  
\begin{lemma}
If $g(a,b,c) = \left((a+b)^{t}+t ca^{t-1}\right)\left(c^{t}+t bc^{t-1}\right)$ has a maximum on \[\setof{(a,b,c)\in \mathbb{R}^3}{a+b+c=1, a,b,c\ge0}\] at a point $(a^*, b^*, c^*)$, then $a^* =0$ or $b^*=0.$
\end{lemma}
\begin{proof}
 We prove this by fixing the second factor of (\ref{eqn:prod}), i.e., $y:=c^{t}+t bc^{t-1}$, letting $c$ vary, and then maximizing the first factor of (\ref{eqn:prod}), i.e., $x:=(a+b)^{t}+{t}ca^{t-1}$.  
Suppose that $y = c^{t}+t bc^{t-1} = K$ for some fixed constant $K$. Notice, then, that the largest possible value of $c$ in this case occurs when $b$ is as small as possible (i.e. $b=0$) and the smallest possible value of $c$ occurs when $b$ is as large as possible (i.e. $a=0$). 
 Since $y=K$ is fixed, we have
\[
	\frac{dy}{dc}=t c^{t-1}\left(1+\frac{db}{dc}\right)+t(t-1)bc^{t-2} = 0.
\]
So, solving for $db/dc$, we get
\[
	\frac{db}{dc}=-1-(t-1)\frac{b}c.
\]
Since $a+b+c=1$, we know that $\frac{da}{dc}+\frac{db}{dc}+\frac{dc}{dc}=0$, and so $\frac{da}{dc}+\frac{db}{dc} = -1$ and  $\frac{da}{dc}=(t-1)\frac{b}c$.  So, differentiating $x$, we get
\begin{align}
	\frac{dx}{dc}&=t(a+b)^{t-1}\left(\frac{da}{dc}+\frac{db}{dc}\right)+t a^{t-1}+t(t-1)ca^{t-2}\frac{da}{dc}\nonumber \\ 
	&=t a^{t-1}+t(t-1)^2 ba^{t-2}-t(a+b)^{t-1}.\label{eqn:dxdc}
\end{align}
The maximum of $x$ (as a function of $c$ along the level curve $y=K$), occurs at either the largest possible $c$ (which happens when $b=0$), the smallest possible $c$ (which happens when $a=0$), or at a value of $c$ where $dx/dc=0$.  Let us rule out this third possibility. We have that $dx/dc=0$ implies
\[
	a^{t-1}+(t-1)^2ba^{t-2}=(a+b)^{t-1}.
\]
Note that if $a\neq 0$, we can divide through by $a^{t-1}$ to get
\[
	1+(t-1)^2\frac{b}a=\left(1+\frac{b}a\right)^{t-1}.
\]
Letting $\lambda=b/a$, we have
\[
	1+(t-1)^2\lambda=(1+\lambda)^{t-1}.
\]
But then by Descartes' rule of signs, this polynomial equation has a unique positive solution at some $\lambda^* \in (0,\infty)$. 

Putting this all together, we have that at the smallest possible $c$, it is the case that $a=0$, and so using (\ref{eqn:dxdc}), we have $dx/dc=-t b^{t-1}<0$.  At the largest possible $c$, we have $b=0$, and so $dx/dc=t a^{t-1}-t a^{t-1}=0$.  We've just shown that at exactly one point between, namely where $b=\lambda^* a$, we have $dx/dc=0$.  And so $x=(a+b)^{t}+{t}ca^{t-1}$, as a function of $c$, starts off decreasing, reaches a point where the derivative is $0$ and so that point in the middle cannot be a maximum.  This shows that the maximum occurs when either $a=0$ or $b=0$.  
\end{proof}

We now know that the maximum in \eqref{eq:hqdef} occurs in $\mathcal{L}_1 = \set{\ell_1, \ell_2}$ where $\ell_1$ is the path that goes up then right, and $\ell_2$ is the path that goes right then up. Note that for $\ell_1$, we have that $I_Y(\ell_1)=0$ and $I_X(\ell_1) = qp^{t-1}.$ For $\ell_2$, we have that $I_X(\ell_2)=0$ and $I_Y(\ell_2) = pq^{t-1}.$
Thus \[h_t(q) = \max\set{(p^t + tqp^{t-1})q^t, p^t(q^t + tpq^{t-1})}.\]
Note that these two functions are symmetric in $q$ and $p = 1-q$ and correspond to $\ell_1$ and $\ell_2$ respectively.
Thus if we find a maximum of the first function at $q=\mu_t$, then the second function has a maximum at $p = \mu_t$.
Let \[f(x) = ((1-x)^t + tx(1-x)^{t-1})x^t = x^t(1-x)^{t-1}(1 + (t-1)x)\] and note that $f(q)$ is the first function in the maximum and $f(1-q)$ is the second.
Then 
\[f'(x) = x^{t-1}(1-x)^{t-2}t\op{-(2t-2)x^2 + (t-2)x +1}\]
and we find that $f'$ has a unique zero on $(0,1)$ at $\mu_t := \frac14\frac{t-2 + \sqrt{t^2 + 4t - 4}}{t - 1}$ and $f(x)$ has a maximum at this point (since $f(0)=f(1)=0$). Thus (referring to \eqref{eq:pitubscaled}) the upper bound in Theorem \ref{thm:pil} is proved. 

Returning to the original interpretation discussed at the beginning of this section, the path $\ell_1$ corresponds to a threshold graph in which $|V_K| = r = pn = (1-q)n$ and $|V_I| = s = qn$ where $a_1,\ldots, a_r = 0$. In other words $d_{\tbar}(v) =0$ for all $v\in V_T$, and so $T$ is complete between $V_K$ and $V_I$. This corresponds to a threshold graph where all the $+$ symbols are to the left of all the $-$ symbols. The maximum of $S_KS_I$ in this case occurs when $q=\mu_t$ and is attained by a threshold graph with code $(+)^{(1-\mu_t) n} (-)^{\mu_t n}$. Similarly for $\ell_2$, we have a threshold graph $T$ where the bipartite graph between $V_K$ and $V_I$ is empty (i.e., a disjoint union of a clique and an independent set). The maximum of $S_KS_I$ in this case occurs when $q=1-\mu_t$ and is attained by the threshold graph with code $(-)^{(1-\mu_t)n}(+)^{\mu_t n} $.

\section{An upper bound on the multicolored product}\label{sec:train}

In this section, we will prove an upper bound on $\prod_{i=1}^r k(G_i)$ where $G_i$ is the graph induced by the $i^\text{th}$ color in an $r$-coloring of the edges of $K_n$.  

\begin{theorem}\label{thm:tuples}
 Let $G_1, \ldots, G_r$ be edge disjoint graphs on the same vertex set V of $n$ vertices. Then the number of  ordered sequences of sets of vertices $(S_1, \ldots, S_r)$ such that $S_j$ induces a clique on $G_j$ for all $j=1,2,\ldots,r$ and $\bigcup_{j=1}^r S_j=V$ is at most
\[
(4r-2)^{r(r-1)} n^{\binom{r}2}.
\]
 \end{theorem}

\begin{proof}
If there are no such sequences, then we are done. So suppose there is at least one such sequence. By omitting repeated vertices, if necessary, we can choose such a sequence where the unions are disjoint, say $T_1,\ldots, T_r.$ If $T_i = V$ for some $i$, then $G_i$ is complete. In this case, there are at most $(n+1)^{r-1} 2^{r-1}$ sequences $(S_1,\ldots, S_r)$ since every other $S_j$ is either a singleton or empty. 

Say that a vertex $v$ is \emph{viable} for graph $G_k$ if it is adjacent, in $G_k$, to all but at most $r-1$ vertices in $T_k$.  Define the directed graph $H$ on the vertex set $[r]=\set{1,\ldots, r}$ by having an edge from $j$ to $k$ if there are at least $2r-1$ vertices in $T_j$ which are viable for $G_k$.

Then there are no vertices $j$ and $k$ in $H$ such that both $j\to k$ and $k\to j$ are edges in $H$. For, otherwise, there would be a subset $U_j$ of $T_j$ of size $2r-1$ such that every $u\in U_j$ is viable for $G_k$.  Likewise, there would be a subset $U_k$ of $T_k$ of size $2r-1$ such that every $u\in U_k$ is viable for $G_j$.  Every vertex which is viable for $G_k$ is adjacent, in $G_k$, to all but at most $r-1$ vertices in $T_k$ and therefore to at least $2r-1-(r-1)=r$ vertices in $U_k$. So $G_k$ contains at least $r(2r-1)$ edges between $U_j$ and $U_k$. But, by the same argument, so does $G_j$.  This is a contradiction as $G_j$ and $G_k$ are edge-disjoint and there are only $|U_j||U_k|=(2r-1)^2$ such edges in total.  Thus there are at most $\binom{r}2$ directed edges in $H$.

Let us first count how many choices there are for $(S_1,\ldots, S_r)$ such that the $S_i$ form a partition of $V$. The main idea is that any such partition must ``mostly'' look like $(T_1,\ldots, T_r)$. Indeed, for any $j\neq k$, $S_j\cap T_k$ must form a clique in both edge disjoint graphs $G_j$ and $G_k$ which implies that $|S_j\cap T_k|\le 1$. So we have that for all $i\in [r]$, $|T_i\setminus S_i|\le r-1$ and $|S_i\setminus T_i| \le r-1$. 
Since $S_i$ contains all but at most $r-1$ elements of $T_i$, and $S_i$ is a clique in $G_i$, each element of $S_i$ must be viable for $G_i$. 

Note that $(S_1,\ldots, S_r)$ is determined if we specify the elements in $S_i\cap T_j$ for  $i\neq j.$ If $v\in S_i\cap T_j$, then $v\in T_j$ and $v$ is viable for $G_i$. 
 If $ji$ is not an edge of $H$, then there are at most $2r-2$ vertices in $T_j$ which are viable for $G_i$ and so there are at most $2r-1$ choices for $S_i\cap T_j$ (including the empty set). If $ji$ is an edge of $H$, then there are at most $|T_j| +1$ choices for  $S_i\cap T_j$ (again including the empty set).
Thus the number of sequences $(S_1,\ldots, S_r)$ is at most
\begin{align*}
	 (2r-1)^{r(r-1)}\prod_{j\to i\in H} (\abs{T_j}+1)\leq (2r-1)^{r(r-1)}n^{\binom{r}2},
\end{align*}
using the facts that $|T_j|\leq n-1$ for all $j\in [r]$ and $e(H)\leq \binom{r}2$.

Now we consider sequences of sets $(S_1,\ldots, S_r)$ which does not necessarily form a partition of $V$. Each such sequence can be obtained by starting with a sequence $(S_1',\ldots, S_r')$ that does form a partition and then adding some subset of $T_i\setminus S_i'$ to $S_i'$.  Thus each such $(S_1'\ldots, S_r')$ gives rise to at most $(2^{|T_i\setminus S_i'|})^r \le 2^{r(r-1)}$ many sequences $(S_1,\ldots, S_r).$ Thus the total number of sequences is at most
\[
	(2r-1)^{r(r-1)}n^{\binom{r}2}\cdot 2^{r(r-1)} = (4r-2)^{r(r-1)}n^{\binom r2}.\qedhere
\]

%


\end{proof}

\begin{corollary}\label{cor:multiprod}
Let $G_1,\ldots, G_r$ be edge-disjoint graphs on the same set of $n$ vertices, $V$. Then 
\[
	\prod_{i=1}^r k(G_i)\leq (4r-2)^{r(r-1)}n^{\stackrel{\binom{r}2}{}} 2^n.
\]
\end{corollary}

\begin{proof}
Firstly, note that $\prod_{i=1}^r k(G_i)$ is simply the number of $r$-tuples $(S_1,S_2,\ldots,S_r)$ where $S_i$ is a clique in $G_i$ for all $i=1,2,\ldots,r$.  If we specify that $\cup_{i=1}^r S_i=S$ for some $S\of V$ with $\abs{S}=k$, then Theorem~\ref{thm:tuples} implies that the number of such tuples is at most
\[
	(4r-2)^{r(r-1)}k^{r(r-1)/2}.
\]
But then,
\begin{align*}
	\prod_{i=1}^r k(G_i)&\leq \sum_{k=1}^n \sum_{\substack{S\subseteq V\\ \abs{S}=k}} (4r-2)^{r(r-1)}k^{r(r-1)/2}\\
	&= (4r-2)^{r(r-1)} \sum_{k=1}^n \binom{n}{k} k^{r(r-1)/2}\\
	&\leq (4r-2)^{r(r-1)} n^{r(r-1)/2}2^n,
\end{align*}
where we used the inequality $\sum_{k}\binom nk k^t \le n^t2^n$ in the final step.
\end{proof}

We are able to construct edge-disjoint graphs $G_1,G_2,\ldots, G_r$ on $n$ vertices where $\prod_{i=1}^r k(G_i)$ matches the bound in Corollary~\ref{cor:multiprod} up to the constant multiple depending on $r$.  Note that the following construction works for any $n$ and $r$, but the bound is much cleaner to state when $n$ is divisible by $\binom{r}2$.  An example of the construction is given in Figure~\ref{fig:con}.

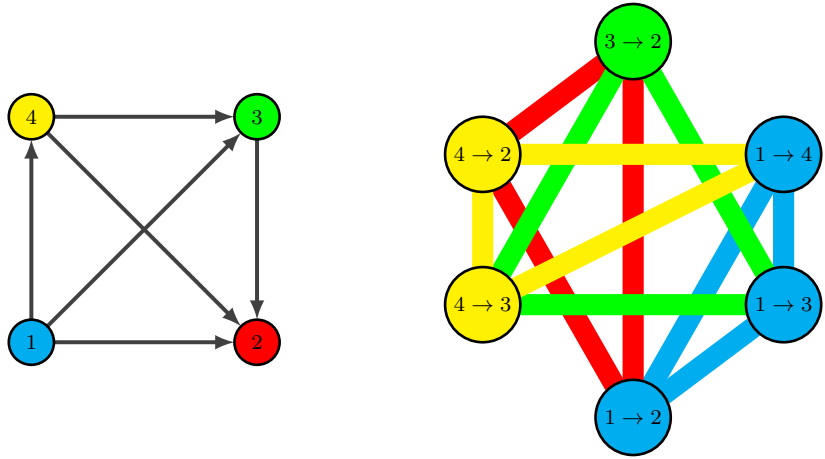
\begin{figure}[h]
\begin{center}
	\begin{tikzpicture}
		\Vertex[color=cyan,label=1]{A}
		\Vertex[color=red,x=3,label=2]{B}
		\Vertex[color=green,x=3,y=3, label=3]{C}
		\Vertex[color=yellow,y=3, label=4]{D}
		\Edge[Direct](A)(B)
		\Edge[Direct](A)(C)
		\Edge[Direct](A)(D)
		\Edge[Direct](C)(B)
		\Edge[Direct](D)(C)
		\Edge[Direct](D)(B)
		\Vertex[x=8,y=-1,size=1,label=1\to 2,Math,color=cyan]{E}
		\Vertex[x=10,y=0.5,size=1,label=1\to 3,Math,color=cyan]{F}
		\Vertex[x=10,y=2.5,size=1,label=1\to 4,Math,color=cyan]{G}
		\Vertex[x=8,y=4,size=1,label=3\to 2,Math,color=green]{H}
		\Vertex[x=6,y=2.5,size=1,label=4\to 2,Math,color=yellow]{I}
		\Vertex[x=6,y=0.5,size=1,label=4\to 3,Math,color=yellow]{J}
		\Edge[lw=8pt,color=cyan](E)(F)
		\Edge[lw=8pt,color=cyan](E)(G)
		\Edge[lw=8pt,color=cyan](G)(F)
		\Edge[lw=8pt,color=red](E)(H)
		\Edge[lw=8pt,color=red](H)(I)
		\Edge[lw=8pt,color=red](E)(I)
		\Edge[lw=8pt,color=green](J)(F)
		\Edge[lw=8pt,color=green](H)(F)
		\Edge[lw=8pt,color=green](H)(J)
		\Edge[lw=8pt,color=yellow](I)(J)
		\Edge[lw=8pt,color=yellow](G)(I)
		\Edge[lw=8pt,color=yellow](G)(J)
	\end{tikzpicture}
\end{center}
	\caption{An example of the construction in Theorem~\ref{thm:const}.  The tournament on the left leads to the coloring of $V$ on the right.}\label{fig:con}
\end{figure}

\begin{theorem}\label{thm:const}
	 If $\binom{r}2$ divides $n$, there exist edge-disjoint graphs $G_1,\ldots, G_r$  on a set of $n$ vertices, say $V$, where 
	 \[
	 	\prod_{i=1}^r k(G_i) \geq \left(\frac{n}{\binom{r}2}\right)^{\binom{r}2}  2^n.
	 \]
\end{theorem}

\begin{proof}

Let $T$ be any tournament on the vertices $1, 2, \ldots, r$. Split the $n$ vertices  of $V$ into $r(r-1)/2$ sets of equal size (or as equal as possible in the case that $\binom{r}2$ does not divide $n$) and label each set by an edge of the tournament, i.e., $S_{i\to j}$ for some edge $i\to j$. Then let $G_i$ consist of all edges within each set $S_{i\to j}$ and all edges between elements of $S_{e_1}$ and $S_{e_2}$ where $e_1$ and $e_2$ are distinct edges both incident to vertex $i$.  Note that the $G_i$ graphs are edge-disjoint.  Also, note that in this construction, we do not color all edges of $K_n$ as we do not assign edges between $S_e$ and $S_f$ for disjoint $e$ and $f$ in $T$ to any of the $G_i$'s.  

If we choose any subset of each $S_{i \to j}$ and at most one vertex from each $S_{j \to i}$, we will get a clique in $G_i$.  Thus the number of cliques in $G_i$ is 
\[
	\prod_{e:e=j\to i} (1+|S_{e}|)\cdot \prod_{f:f=i\to k} 2^{|S_{f}|}.
\]
If we multiply these factors over all $i$, each edge of $T$ counts once towards each of the products above.  And so,
\begin{align*}
	\prod_{i=1}^r k(G_i)&= \prod_{e\in T} (1+ |S_e|) 2^{|S_{e}|}\\
	&= 2^n\prod_e (1+ |S_e|)\\
	&\geq 2^n \prod_e \frac{n}{\binom{r}2}\\
	&=2^n\left(\frac{n}{\binom{r}2}\right)^{\binom{r}2}.\qedhere
\end{align*}
\end{proof}

\section{Conclusion}
In this paper we have explored upper and lower bounds on the product and sum of the number of independent sets in a graph and its complement. More generally, we have also investigated the  fixed size and multicolor versions of these problems. Many interesting open problems remain. 

In Theorem \ref{thm:pil} we proved a tight upper bound on $\pi_t(G) = k_t(G_1)k_t(G_2)$ where $G_2=\overline{G_1}$. We made use of compression to show that the $G$ which attains the maximum must be a threshold graph. One may pursue a multicolor version of this theorem, i.e., finding a (tight) upper bound for $\prod_{k=1}^r k_t(G_i)$. As mentioned earlier, we know of no multicolor analogue of graph compression which makes the problem much more difficult. We trivially have an upper bound of $\binom{n}{t}^r\sim \frac{1}{t!^r}n^{tr}$. Partitioning $K_n$ into $r$ vertex disjoint cliques shows that there is a graph with $\prod_{k=1}^r k_t(G_i) \ge \binom{n/r}{t}^r \sim \frac{1}{r^{tr}t!^r}n^{tr}$. 

Another interesting place for improvement is in the lower bound for the multicolor product $\prod_{i=1}^r k(G_i)$ in Theorem \ref{thm:r-col-lb}. In this case, the proven lower bound and the tightness example differ by a factor of $3\log_3 r$ in the exponent. While finding the exact exponent may be a difficult problem, it may be nice to find a tightness example which differs only by a factor not depending on $r$. 

Finally, a possible area of future exploration is that of a lower bound on the multicolor fixed-size sum $\sum_{i=1}^r k_t(G_i)$.  The case of $r=3$ colors and triangles ($t=3$) is relatively well-known and was settled by Cummings et al.~\cite{CKPSTY}.  The authors of that paper also asked about triangles in graphs colored with more than three colors.

\bibliographystyle{amsplain}
\bibliography{nordgad}

\end{document}